\documentclass[12pt]{amsart} 
\usepackage[letterpaper, margin = .8in]{geometry} 

\usepackage{lmodern} 

\usepackage{amsmath, amsthm, amssymb, amsfonts} 
\usepackage{mathtools} 
\usepackage{booktabs} 
\usepackage{enumitem} 
\usepackage{comment} 
\usepackage{todonotes}

\usepackage{graphicx} 
\usepackage{tikz}
\usepackage[colorlinks]{hyperref} 
\hypersetup{
	linkcolor=blue,
	citecolor=blue,
	filecolor=red,
	urlcolor=blue,
	menucolor=red,
	runcolor=red}

\newcommand{\F}{\mathbb{F}}
\newcommand{\C}{\mathcal{C}} 
\newcommand{\T}{\mathcal{T}}

\newcommand{\Prob}{\mathbb{P}}

\newcommand{\Hh}{\mathcal{H}}

\newtheorem{thm}{Theorem}
\newtheorem{lem}{Lemma}

\newtheorem{them}{Theorem}

\theoremstyle{definition}
\newtheorem{defn}{Definition}

\newtheorem*{obs}{Observation}

\newcommand{\R}{\mathbb{R}}

\newcommand{\Pole}{\operatorname{Pole}}
\newcommand{\sm}{\setminus}

\title{VC-Dimension and Distance Chains in $\mathbb{F}_q^d$}
\author[\scalebox{0.8}{Ascoli, Betti, Cheigh, Iosevich,  Jeong, Liu, McDonald, Milgrim, Miller, Romero Acosta, Velazquez Iannuzzelli}]{ Ruben Ascoli, Livia Betti, Justin Cheigh, Alex Iosevich, Ryan Jeong, Xuyan Liu, Brian McDonald, Wyatt Milgrim, Steven J. Miller, Francisco Romero Acosta, Santiago Velazquez Iannuzzelli}

\thanks{The fourth listed author's research was supported in part by the National Science Foundation grant no. HDR TRIPODS - 1934962 and the National Science Foundation grant DMS 2154232 }

\begin{document}

\maketitle
\begin{abstract}
Given a domain $X$ and a collection $\mathcal{H}$ of functions $h:X\to \{0,1\}$, the Vapnik-Chervonenkis (VC) dimension of $\mathcal{H}$ measures its complexity in an appropriate sense. In particular, the fundamental theorem of statistical learning says that a hypothesis class with finite VC-dimension is PAC learnable.  Recent work by Fitzpatrick, Wyman, the fourth and seventh named authors studied the VC-dimension of a natural family of functions $\mathcal{H}_t^{'2}(E): \F_q^2\to \{0,1\}$, corresponding to indicator functions of circles centered at points in a subset $E\subseteq \mathbb{F}_q^2$. They showed that when $|E|$ is large enough, the VC-dimension of $\mathcal{H}_t^{'2}(E)$ is the same as in the case that $E = \mathbb F_q^2$.  
We study a related hypothesis class, $\Hh_t^d(E)$, corresponding to intersections of spheres in $\mathbb{F}_q^d$, and ask how large $E\subseteq \mathbb{F}_q^d$ needs to be to ensure the maximum possible VC-dimension. We resolve this problem in all dimensions, proving that whenever $|E|\geq C_dq^{d-1/(d-1)}$ for $d\geq 3$, the VC-dimension of $\Hh_t^d(E)$ is as large as possible. We get a slightly stronger result if $d=3$: this result holds as long as $|E|\geq C_3 q^{7/3}$. Furthermore, when $d=2$ the result holds when $|E|\geq C_2 q^{7/4}$.
\end{abstract}


\section{Introduction}
Recent work has emerged studying the Vapnik-Chervonenkis (VC) dimension of certain classes of functions on vector spaces in finite fields, notably \cite{FIMW} and \cite{IMS}. For a collection $\Hh$ of functions   $h: \F_q^d\to\{0,1\}$, the VC-dimension measures the complexity of the system from the point of view of learning theory. We give a brief overview of the connection with PAC learning in Section 2. 
 For an introduction to the subject, see for example \cite{shalev2014understanding}.

\begin{defn}[Shattering]
    Let $X$ be a set, and let $\mathcal{H}$ be a collection of functions from $X$ to $\{0,1\}$. Then, $\mathcal{H}$ \textit{shatters} a finite set $C \subset X$ if the restriction of $\mathcal{H}$ to $C$ yields all possible functions from $C$ to $\{0,1\}$.
\end{defn}

\begin{defn}[VC-dimension]
    Let $X$ be a set, and let $\mathcal{H}$ be a collection of functions from $X$ to $\{0,1\}$. Then, $\mathcal{H}$ has \textit{VC-dimension} $n$ if there exists a set $C \subset X$ of size $n$ that is shattered by $\mathcal{H}$, and no subset of size $n+1$ of $X$ is shattered by $\mathcal{H}$. That is, the VC-dimension of $\Hh$ is the maximal size of a set it can shatter. 
\end{defn}
For a domain $X$, we will refer to the functions $h:X\to \{0,1\}$ as classifiers, and a collection $\mathcal{H}$ of such functions as a hypothesis class.

Let $q$ be a power of an odd prime, and let $\F_q^d$ be the $d$-dimensional vector space over the finite field with $q$ elements. 
Throughout this paper, for $x\in \F_q^d$, we use $||x||$ to mean $x_1^2+\ldots+x_d^2$.
We do not take a square root since not every element $\mathbb{F}_q$ is a square.  Consider the distance graph $\mathcal{G}_t(E)$, whose vertices are points in $E\subseteq \mathbb{F}_q^d$ with an edge $x\sim y$ whenever $||x-y||=t$.  There has been extensive work on configuration problems over finite fields in the following sense:  given a graph $G$, one seeks to find an exponent $\alpha<d$ and a constant $C>0$ so that for any $E\subseteq \mathbb{F}_q^d$ with $|E|\geq Cq^{\alpha}$, $q$ sufficiently large, there is an embedding of $G$ in $\mathcal{G}_t(E)$.

For the simplest case, where $G$ is just one edge, the fourth author and Rudnev established the exponent $\alpha=\frac{d+1}{2}$ in \cite{IR}.  Since then such results have been achieved for many other graphs; for example, Bennett, Chapman, Covert, Hart, the fourth author, and Pakianathan achieved the same exponent $\frac{d+1}{2}$ for paths of arbitrary length in \cite{BCCHIP}.  In \cite{IJM}, the fourth and seventh authors and Jardine obtained cycles of length $n\geq 4$ when $d\geq 3$, and cycles of length $n\geq 5$ when $d\geq 2$, with exponent ranging from $\frac{d+1}{2}$ to $\frac{d+2}{2}$ depending on the length of the cycle.  These graphs discussed so far are all rather sparse, and indeed these problems are generally harder for graphs with many edges.  On the other end of the spectrum, in \cite{IP} the fourth author and Parshall obtained a general result for any graph $G$, with exponent $\frac{d-1}{2}+t$, where $t$ is the maximum edge degree of $G$.  So for example if $G=K_n$ is the complete graph on $n$ vertices, then in order for this to yield a nontrivial result, the dimension must be at least $2n-2$.  

In the hypothesis class we define below, showing that the VC-dimension of $H_t^d(E)$ is equal to $d$ is equivalent to constructing a particular graph $G$ embedded in $\mathcal{G}_t(E)$.  However, since the graph $G$ that we need to construct depends on the dimension $d$, and in particular the maximum vertex degree is also $d$, we cannot apply the result from \cite{IP} because $\frac{d-1}{2}+t$ will never be small enough.  This leads to a configuration problem requiring a new approach.

\vspace{0.125 in}

\subsection{Main results}

\begin{defn}
We define the following hypothesis class with respect to a set $E \subset \F_q^d$:
\begin{equation}
    \mathcal{H}^d_t(E) = \left\{h_{u,v} (x): (u,v) \in E \times E, u \neq v \right\},
\end{equation}
where $h_{u,v} : E \rightarrow \{0,1\}$ is defined by \begin{equation}
    h_{u,v} (x) = \begin{cases}
    1 \mbox{ if $||x-u|| = ||x-v|| = t$} \\
    0 \mbox{ otherwise}.
    \end{cases}
\end{equation}
In the case where $E = \F_q^d$ we use $\mathcal{H}_t^d$ rather than $\mathcal{H}_t^d(\F_q^d)$.
\end{defn}

These classifiers are directly inspired by those studied in \cite{FIMW}. In that paper Fitzpatrick, Wyman, and the fourth and seventh authors studied an analogous set of classifiers with only one parameter, namely
\begin{equation}
    \mathcal{H}_t^{'d} = \left\{h_y (x): y \in \mathbb{F}_q^d\right\},
\end{equation}
where 
\begin{equation}
    h_y (x) = \begin{cases}
    1 \mbox{ if $||x-y|| = t$} \\
    0 \mbox{ otherwise},
    \end{cases}
\end{equation}
with an analogous definition of $\Hh_t^{'d}(E)$. Since $d+1$ points determine a $d$-dimensional sphere the VC-dimension of $\Hh_t^{'d}(E)$ is at most $d+1$. They showed in the case of $d=2$ that whenever $|E|\geq Cq^{15/8}$, for some constant $C$, the VC-dimension of $\Hh_t^{'2}(E)$ is equal to $3$, the largest it could be. However, they were unable to extend this result to higher dimensions, and even the $d=3$ case is an open problem. For the classifiers we study, however, we obtain results for all dimensions $d\geq 2$.  Our main result is as follows:

\begin{thm}\label{main thm}
    If $E \subset \mathbb{F}_q^d$, $d \geq 2$, and
    \begin{equation}
        |E| \geq \left\{\begin{array}{ll} 
        C q^{\frac{7}{4}} & d=2  \\      
        C q^{\frac{7}{3}}
     &  d=3 \\
        C q^{d-\frac{1}{d-1}} & d\geq 4
        \end{array}\right.
    \end{equation}
    for a constant $C$ depending only on $d$, then the VC-dimension of $\mathcal{H}_t^d (E)$ is equal to $d$.
\end{thm}

It's easy to see that the VC-dimension of $\mathcal{H}_t^d(E)$ cannot be greater than $d$. This is because $d+1$ points determine a unique $d$-dimensional sphere, so it is not possible to find $d+1$ points such that there are two distinct points distance $t$ away from all of them.

\section{Motivation: Connections to Learning Theory}

The study of the VC-dimension of the classifiers over $\F_q^d$ introduced here (as well as those corresponding to spheres in $\F_q^2$ and hyperplanes in $\F_q^3$, as studied in \cite{FIMW} and \cite{IMS}, respectively) can be motivated from the perspective of computational learning theory, where one is broadly interested in learning concepts with low error with high probability. We begin by introducing the relevant notions more generally, then discuss them specifically within the present context; see \cite{kearns1994introduction, shalev2014understanding} for a more thorough treatment of VC-dimension and its relevance to PAC theory.

For what follows, fix a set $E$ and a hypothesis class of functions $\mathcal H$ from $E$ to $\{0,1\}$, and consider the learning task associated with $\mathcal{H}$. Fix a classifier $c \in \mathcal H$, which is the classifier the learner would like to learn, and a probability distribution $\mathcal D$ over $E$, which is unknown to the learner. The learner is incrementally given access to values of the function $c(x)$, with input $x \in E$ drawn i.i.d. from the distribution $\mathcal D$. Generally one desires an algorithm which takes these sampled values of $c(x)$ as input, and returns a classifier $h\in H$ which is close to the true classifier $c$ in an appropriate sense, with high probability.

\vspace{0.125 in}

More precisely, define the loss function $L_{\mathcal D,c}: \mathcal H \to [0,1]$ by 
\begin{align}
    L_{\mathcal D,c}(h) = \Prob_{x \sim \mathcal D}\left[ c(x) \neq h(x) \right]
\end{align}
where $x \sim \mathcal D$ denotes that $x$ is drawn from the distribution $\mathcal D$. The notion of learnability illustrated above is captured precisely by the following definition.
\begin{defn}
The hypothesis class $\mathcal H$ is \textbf{PAC learnable} if there exists a function
\begin{align*}
    m_{\mathcal H}: (0,1)^2 \to \mathbb N
\end{align*}
and an algorithm $\mathcal A$ such that given any $\epsilon, \delta \in (0,1)$, distribution $\mathcal D$ over $E$, and classifier $c \in \mathcal H$, $\mathcal A$ chooses $h \in \mathcal H$ satisfying $L_{\mathcal D,c}(h) \leq \epsilon$ with probability at least $1-\delta$ when given $m \geq m_{\mathcal H}(\epsilon,\delta)$ i.i.d. samples from $\mathcal D$ and their mappings under $c$.
\end{defn}
The following theorem is a quantitative version of the fundamental theorem of machine learning, and provides the link between VC-dimension and learnability.
\begin{them} \label{thm:FTML}
The hypothesis class $\mathcal H$ has finite VC-dimension if and only if $\mathcal H$ is PAC learnable. Furthermore, if the VC-dimension of $\mathcal{H}$ is equal to $n$, then there exist constants $C_1, C_2$ such that
\begin{align}
    C_1 \frac{n+\log\left( \frac{1}{\delta} \right)}{\epsilon} \leq m_{\mathcal H}(\epsilon, \delta) \leq C_2 \frac{n\log\left( \frac{1}{\epsilon} \right) + \log \left( \frac{1}{\delta} \right)}{\epsilon}
\end{align}
in an algorithm with respect to which $\mathcal H$ is PAC learnable.
\end{them}
We now consider the learning task associated with our classifiers $\mathcal H_t^d(E)$ for $d\geq 3$. For a fixed nonzero $t \in \mathbb{F}_q$ and a distribution $\mathcal D$ over $E\subset \F_q^d$, the learner aims to construct a classifier $h: E \to \{0,1\}$ that maps $x \in E$ to $1$ if $x$ is on the intersection of two fixed spheres of radius $t$ centered at points $u \neq v$ unknown to the learner. Theorem \ref{thm:FTML} tells us that since the VC-dimension of $\mathcal H_t^d(E)$ is finite $\mathcal H_t^d(E)$ is PAC learnable. Towards a stronger understanding, let us assume $E = \F_q^d$, that is, we consider $\Hh_q^d$, and let $\mathcal D$ be the uniform distribution over $\F_q^d$. The intersection of two spheres of non-zero radius in $\F_q^d$ has size $q^{d-2} + o(q^{d-2})$ \cite{HDISU}, so we have that for all $h \in \mathcal H_t^d$,
\begin{align}
    L_{\mathcal D,c}(h) \leq \frac{2}{q^2}\left(1+o(1)\right)
\end{align}
so one must choose $\epsilon < \frac{2}{q^2}$ for meaningful results; choosing $\delta = \epsilon < \frac{1}{q^2}$ and referring to Theorem \ref{thm:FTML} yields that we must consider random samples of size at most $Cq^2\log\left( q^2 \right)$, for some constant $C > 0$. Furthermore, since $d-1$ points determine a $(d-2)$-dimensional sphere (i.e. the intersection of two spheres in $\F_q^d$), for large $q$, we only need $\epsilon$ slightly less than $\frac{2}{q^2}$ to get $L_{\mathcal D,c}(h)=0$.

\section{Preliminaries}\label{Prelim}
The authors of \cite{FIMW} noted that the problem they were studying was most productively thought of in the context of point configurations. Recall that their classifiers were the functions $\{h_y: y\in E\}$ where $h_y(x) = 1$ if $||y-x||=t$ and $h_y(x)$ otherwise. Thus shattering a set of size $n$ means finding sets $A,B\subset E$ such that $|A|=n$ and for each $S\subset A$ we can find a $b_S \in B$ such that for each $a\in A$ we have $||a-b_S||=t$ if and only if $a\in S$. These points taken together form a point configuration which can be thought as a subgraph of the distance graph $\mathcal{G}_t(E)$. They leverage the estimate on the number of edges in $\mathcal{G}_t(E)$ from \cite{IR}, along with an argument pigeonholing on the directions of such edges, to construct the desired configuration. These theorems have a geometric flavor to them: the similarity of $||\cdot||$ to the standard norm on Euclidean space means many familiar results concerning the geometry of $\R^d$ carry over to $\F_q^d$.  In particular, the spheres defined by our notion of distance have similar intersection properties to spheres in $\mathbb{R}^n$.

\begin{defn}
    Let $S_t = \{ x \in \mathbb{F}_q^d : ||x|| = t \}$ For notational convenience, we often identify a set with its indicator function, so that $S_t(x)=1$ precisely when $||x||=t$.  
\end{defn}

Since our work concerns a variant of the problem in \cite{FIMW}, we will follow a similar approach. First note that to shatter $n$ points it is necessary to find points $\{x^1,...,x^n\}$ and points $\{y,z\}$ such that $||x^i-y|| = ||x^i - z|| = t$ for $1\leq i \leq n$\footnote{We use the notation $x^i$ instead of $x_i$ following the convention in \cite{FIMW}. This superscript should be read as a kind of index, not an exponent.}. This leads us to the following natural definitions.

\begin{defn}[Prism]
The $(n+2)$-tuple $P = (y,z,x^1, \ldots, x^n) \in (\mathbb{F}_q^d)^{n+2}$ is an $n$-prism if for all $i \leq n,$ $||x^i - y|| = ||x^i - z|| = t$. The \textit{tail} of $P$, denoted $\mathcal{T}(P)$, is the set $\{y,z\}$. The \textit{center} of $P$, denoted $\mathcal{C}(P)$, is the set $\{x^1,x^2,\dots,x^n\}$. We may also write $P = (\mathcal{T}, \mathcal{C})$.
\end{defn}

Below we have an $n$-prism\footnote{Note that this distance graph is isomorphic to the complete bipartite graph $K_{2,n}$.} $(y,z,x^1, \ldots, x^n)$, as seen in the distance graph of $\mathbb{F}_q^d$.

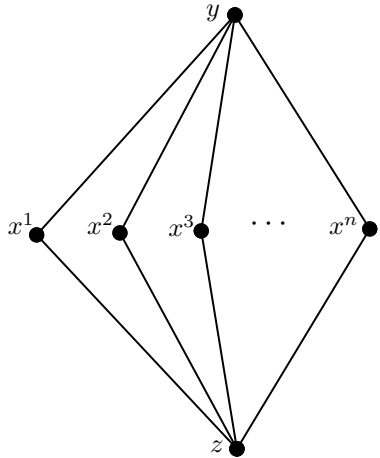
\begin{center}

\tikzset{every picture/.style={line width=0.75pt}} 

\begin{tikzpicture}[x=0.75pt,y=0.75pt,yscale=-1,xscale=1]

\draw    (200,131) -- (300,20) ;
\draw [shift={(300,20)}, rotate = 312.02] [color={rgb, 255:red, 0; green, 0; blue, 0 }  ][fill={rgb, 255:red, 0; green, 0; blue, 0 }  ][line width=0.75]      (0, 0) circle [x radius= 3.35, y radius= 3.35]   ;
\draw [shift={(200,131)}, rotate = 312.02] [color={rgb, 255:red, 0; green, 0; blue, 0 }  ][fill={rgb, 255:red, 0; green, 0; blue, 0 }  ][line width=0.75]      (0, 0) circle [x radius= 3.35, y radius= 3.35]   ;
\draw    (301,239) -- (200,131) ;
\draw [shift={(200,131)}, rotate = 226.92] [color={rgb, 255:red, 0; green, 0; blue, 0 }  ][fill={rgb, 255:red, 0; green, 0; blue, 0 }  ][line width=0.75]      (0, 0) circle [x radius= 3.35, y radius= 3.35]   ;
\draw [shift={(301,239)}, rotate = 226.92] [color={rgb, 255:red, 0; green, 0; blue, 0 }  ][fill={rgb, 255:red, 0; green, 0; blue, 0 }  ][line width=0.75]      (0, 0) circle [x radius= 3.35, y radius= 3.35]   ;
\draw    (242,130) -- (300,20) ;
\draw [shift={(300,20)}, rotate = 297.8] [color={rgb, 255:red, 0; green, 0; blue, 0 }  ][fill={rgb, 255:red, 0; green, 0; blue, 0 }  ][line width=0.75]      (0, 0) circle [x radius= 3.35, y radius= 3.35]   ;
\draw [shift={(242,130)}, rotate = 297.8] [color={rgb, 255:red, 0; green, 0; blue, 0 }  ][fill={rgb, 255:red, 0; green, 0; blue, 0 }  ][line width=0.75]      (0, 0) circle [x radius= 3.35, y radius= 3.35]   ;
\draw    (283,129) -- (300,20) ;
\draw [shift={(300,20)}, rotate = 278.86] [color={rgb, 255:red, 0; green, 0; blue, 0 }  ][fill={rgb, 255:red, 0; green, 0; blue, 0 }  ][line width=0.75]      (0, 0) circle [x radius= 3.35, y radius= 3.35]   ;
\draw [shift={(283,129)}, rotate = 278.86] [color={rgb, 255:red, 0; green, 0; blue, 0 }  ][fill={rgb, 255:red, 0; green, 0; blue, 0 }  ][line width=0.75]      (0, 0) circle [x radius= 3.35, y radius= 3.35]   ;
\draw    (301,239) -- (242,130) ;
\draw [shift={(242,130)}, rotate = 241.57] [color={rgb, 255:red, 0; green, 0; blue, 0 }  ][fill={rgb, 255:red, 0; green, 0; blue, 0 }  ][line width=0.75]      (0, 0) circle [x radius= 3.35, y radius= 3.35]   ;
\draw [shift={(301,239)}, rotate = 241.57] [color={rgb, 255:red, 0; green, 0; blue, 0 }  ][fill={rgb, 255:red, 0; green, 0; blue, 0 }  ][line width=0.75]      (0, 0) circle [x radius= 3.35, y radius= 3.35]   ;
\draw    (283,129) -- (301,239) ;
\draw [shift={(301,239)}, rotate = 80.71] [color={rgb, 255:red, 0; green, 0; blue, 0 }  ][fill={rgb, 255:red, 0; green, 0; blue, 0 }  ][line width=0.75]      (0, 0) circle [x radius= 3.35, y radius= 3.35]   ;
\draw [shift={(283,129)}, rotate = 80.71] [color={rgb, 255:red, 0; green, 0; blue, 0 }  ][fill={rgb, 255:red, 0; green, 0; blue, 0 }  ][line width=0.75]      (0, 0) circle [x radius= 3.35, y radius= 3.35]   ;
\draw    (368,128) -- (300,20) ;
\draw [shift={(300,20)}, rotate = 237.8] [color={rgb, 255:red, 0; green, 0; blue, 0 }  ][fill={rgb, 255:red, 0; green, 0; blue, 0 }  ][line width=0.75]      (0, 0) circle [x radius= 3.35, y radius= 3.35]   ;
\draw [shift={(368,128)}, rotate = 237.8] [color={rgb, 255:red, 0; green, 0; blue, 0 }  ][fill={rgb, 255:red, 0; green, 0; blue, 0 }  ][line width=0.75]      (0, 0) circle [x radius= 3.35, y radius= 3.35]   ;
\draw    (301,239) -- (368,128) ;
\draw [shift={(368,128)}, rotate = 301.12] [color={rgb, 255:red, 0; green, 0; blue, 0 }  ][fill={rgb, 255:red, 0; green, 0; blue, 0 }  ][line width=0.75]      (0, 0) circle [x radius= 3.35, y radius= 3.35]   ;
\draw [shift={(301,239)}, rotate = 301.12] [color={rgb, 255:red, 0; green, 0; blue, 0 }  ][fill={rgb, 255:red, 0; green, 0; blue, 0 }  ][line width=0.75]      (0, 0) circle [x radius= 3.35, y radius= 3.35]   ;

\draw (306,120.4) node [anchor=north west][inner sep=0.75pt]    {$\cdots $};
\draw (184,118.4) node [anchor=north west][inner sep=0.75pt]  [font=\footnotesize]  {$x^{1}$};
\draw (224,118.4) node [anchor=north west][inner sep=0.75pt]  [font=\footnotesize]  {$x^{2}$};
\draw (265,119.4) node [anchor=north west][inner sep=0.75pt]  [font=\footnotesize]  {$x^{3}$};
\draw (346,120.4) node [anchor=north west][inner sep=0.75pt]  [font=\footnotesize]  {$x^{n}$};
\draw (284,13.4) node [anchor=north west][inner sep=0.75pt]  [font=\footnotesize]  {$y$};
\draw (286,233.4) node [anchor=north west][inner sep=0.75pt]  [font=\footnotesize]  {$z$};

\end{tikzpicture}
\end{center}

\begin{defn}[Non-Degenerate Prism]
    We call an $n$-prism $P = (y,z,x^1, \ldots, x^n)$ \textit{non-degenerate} if all of its components are distinct points.
\end{defn}

Since our goal is specifically to show that the VC-dimension of $\Hh_t^d(E)$ is $d$, we are only interested in $d$-prisms and thus will henceforth use term ``prism'' to refer to a ``$d$-prism'' interchangeably.  

We also frequently find it useful to refer to all the points that are distance $t$ away from some given set $A$, for instance when looking for a classifier that can specify $A$. 
\begin{defn}[Pole]
    We say a point $y\in E$ is a \textit{pole} of the set $A\subset E$ if
    \begin{equation}
        y \in \bigcap_{x \in A} (S_t + x).
    \end{equation}
    We denote the set of poles of $A$ as $\Pole(A)$.
\end{defn}

These definitions give us a way to attack our central problem of shattering $d$ points. In particular:

\begin{obs}
If we can find a nondegenerate $d$-prism $P$ in $E$ such that for each $A \subsetneq \C(P)$ we can find a point $y(A) \in \Pole(A)$ with the property that $y(A)$ is not distance $t$ from any $c \in \mathcal{C}\sm A$ then the $VC$-dimension of $\Hh_t^d$ is $d$. Specifically, $\Hh_t^d(E)$ can shatter $\C(P)$.
\end{obs}

To see why this is true, note that if $\T(P)=\{z,w\}$, then we can specify any $A\subsetneq \C(P)$ with the classifier $h_{z,y(A)}$.  Furthermore, we can specify the whole $\C(P)$ with $h_{z,w}$. We will show that such a prism exists by counting the number of $d$-prisms and then applying the Pigeonhole Principle on the number of $d$-prisms that do not have this property. We define the following.

\begin{defn}[$P$-Bad set]
    Fix a $d$-prism $P = (\mathcal T, \mathcal C)$ with center $\mathcal{C}$. A subset $A \subset \mathcal{C}$ is $P$-bad, or bad in $P$, if 
    \begin{equation}
        \bigcap_{x \in A} (S_t + x) \ \subset \ \bigcup_{y \in \mathcal{C} \setminus A} (S_t + y).
    \end{equation}
\end{defn}

We say a set is bad if it is $P$-bad for some prism. We say that a prism $P$ \textit{admits a bad set} if there is some subset $A\subset \C(P)$ that is $P$-bad. Note that our problem reduces to finding a nondegenerate prism that does not admit a bad set. As it turns out, our proof will require us to further restrict ourselves to only considering nondegenerate prisms with affinely independent centers.

\begin{defn}
We say that a prism is \textit{affinely nondegenerate} if it is nondegenerate and its center is affinely independent. We say that a prism is \textit{affinely degenerate} if it is nondegenerate but not affinely nondegenerate. 
\end{defn}

\section{Proof of Theorem \ref{main thm}}
We wish to find a nondegenerate prism that does not admit a bad set. We begin by obtaining a lower bound on the total number of nondegenerate prisms. We do this by noting that a nondegenerate prism is just a choice of $d$ distinct paths of length $2$ between two distinct points. The total number of such $2$-paths in $E$ is a special case of Theorem 1.1 in \cite{BCCHIP}:

\begin{them}\label{chain count}
    Let $E \subset \mathbb{F}_q^d$, where $d \geq 2$ and $|E| > \frac{2k}{\log 2} q^{\frac{d+1}{2}}$. Suppose that $t \neq 0$. Define
    \begin{equation}
        \Gamma_k = |\{(x^1, \ldots, x^{k+1}) \in E \times \cdots \times E : ||x^i - x^{i+1}|| = t, \ 1 \leq i \leq k \}|.
    \end{equation}
    Then,
    \begin{equation}
        \Gamma_k = \frac{|E|^{k+1}}{q^k} + \mathcal{D}_k \quad where \quad |\mathcal{D}_k| \leq \frac{2k}{\log 2} q^{\frac{d+1}{2}} \frac{|E|^k}{q^k}. 
    \end{equation}
    In particular, for $E$ satisfying the hypotheses of Theorem \ref{main thm}, 
    \begin{equation}
        \Gamma_2 \geq \frac{|E|^3}{2q^2}.
    \end{equation}
\end{them}

This allows us to obtain the following theorem.

\begin{thm}\label{prism count}
    Let $E \subset \mathbb{F}_q^d$, $d \geq 3$. Let $N_d (E)$ be the number of non-degenerate $d$-prisms in $E$. If $|E| > \frac{4}{\log(2)} q^{\frac{d+1}{2}}$, then\footnote{We use the notation $A\gtrsim B$ to indicate that for some constant $c$, $A\geq cB$. We use $\gtrsim_d$ to indicate that the constant $c$ may depend on $d$. Throughout this paper, we assume that $d\ll q$ --- that is, $d$ is treated as a constant.}
    \begin{align}
        N_d (E) \ \gtrsim_d \ \frac{|E|^{d+2}}{q^{2d}}.
    \end{align}
\end{thm}
\begin{proof}
Let $k_{(x,y)}$ be the number of paths of length 2 from $x$ to $y$ in the distance graph of $E$. Then,

\begin{align}\label{np3} N_d (E) \ = \ \sum_{\substack{x,y \in E \\ x \neq y}} k_{(x,y)}(k_{(x,y)}-1)\cdots(k_{(x,y)}-d+1).
\end{align}  
For each $(x,y)\in E^2$, define \begin{equation}
    k'_{(x,y)}= \max(k_{(x,y)}-d+1, 0).
\end{equation}
Note that Equation \ref{np3} implies that
\begin{align}\label{np3 part 2} N_d (E) \ \geq \ \sum_{\substack{x,y \in E \\ x \neq y}} (k_{(x,y)}')^d.  
\end{align}  

Theorem \ref{chain count} gives us a lower bound on the total number paths of length $2$ in the distance graph of $E$ when $|E| > \frac{4}{\log(2)} q^{\frac{d+1}{2}}$: \begin{align} \sum_{x,y \in E} k_{(x,y)} \ \gtrsim \ |E|^3q^{-2}.
\end{align}
Now note that the number of 2-paths where the endpoints are the same is just twice the number of 1-paths. By Theorem \ref{chain count}, $\Gamma_1 \lesssim |E|^2q^{-1}$. Thus
\begin{align}\label{c2} \sum_{\substack{x,y \in E \\ x \neq y}} k_{(x,y)} \ \gtrsim \ |E|^3q^{-2} .
\end{align}
Then,
\begin{align}\label{c2.2}
\sum_{\substack{x,y \in E \\ x \neq y}} k_{(x,y)}' \ \geq \ \sum_{\substack{x,y \in E \\ x \neq y}} (k_{(x,y)}-d+1) & \ \geq \ |E|^3q^{-2} - (d-1)|E|^2 \ \gtrsim_d \ |E|^3q^{-2},
\end{align} 
where we have used $|E|^3q^{-2}\gg (d-1)|E|^2$ to bound $(d-1)|E|^2$ by a small constant times $|E|^3q^{-2}$. H\"older's inequality states that for nonnegative $a_i$, $b_i$ and positive $r,s$,
\[\left(\sum_{i=1}^n a_i^rb_i^s\right)^{r+s} \ \leq \ \left(\sum_{i=1}^{n}a_i^{r+s}\right)^r\left(\sum_{i=1}^{n}b_i^{r+s}\right)^s.\]
Setting $n=|E|^2$, $a_i=k_{(x,y)}'$ (where we arbitrarily index the pairs $(x,y)$), $b_i=1$, $r=1$, and $s=d-1$, we get that 
\[\left(\sum_{\substack{x,y \in E \\ x \neq y}} k_{(x,y)}'\right)^d \ \leq \ \left(\sum_{\substack{x,y \in E \\ x \neq y}}(k_{(x,y)}')^d\right)\cdot (|E|^2)^{d-1},\] or (by Equation \ref{c2.2})
\[\left(\sum_{\substack{x,y \in E \\ x \neq y}}(k_{(x,y)}')^d\right) \ \gtrsim_d \  \left(\frac{
|E|^3}{q^2}\right)^d \cdot  \frac{1}{|E|^{2d-2}} = \frac{|E|^{d+2}}{q^{2d}}.\]
By Equation \ref{np3 part 2}, the proof is complete.
\end{proof}

Our goal now is to show that a positive proportion of these prisms are affinely nondegenerate. We require two intermediary results. The first estimates the number of points on a sphere in $\F_q^d$. See for example the appendix of \cite{BHIPR17} for a treatment of a theorem proved by Minkowski \cite{Min11} at the age of 17.  The following is a special case.

\begin{them}\label{size of a sphere}
Let the sphere $S_t \subset \F_q^d$ be as defined above. Then 
\begin{equation}
    q^{d-1} - q^{\frac{d}{2}} < |S_t^d| < q^{d-1} + q^{\frac{d}{2}}.
\end{equation}
\end{them}

Note that this result says a $d$-sphere contains approximately (and asymptotically) $q^{d-1}$ points.

We will also need the following lemma:
\begin{lem}\label{Brian Lemma}
    Let $A$ be a $n$-dimensional affine subspace of $\F_q^d$. Then $|A\cap S_t| \leq 2q^{n-1}$.
\end{lem}
\begin{proof}
We can write each element $a$ of $A$ as $b + w$ where $b\in A$ is a fixed basepoint, and $w \in V$, a $n$-dimensional linear subspace. Choose a basis $v_1,...,v_n$ for $V$, then any $a\in A$ can be written $b + c_1v_1 + c_2v_2 + ... + c_nv_n$. We show that once $c_1, ..., c_{n-1}$ have been fixed there are at most two choices for $c_n$ such that $||a||=t$. Since there are $q^{n-1}$ choices for $(c_1,...,c_{n-1})$ it follows that $|A\cap S_t| < 2q^{n-1}$.

For notational convenience let $v_0 = b$ and $c_0 = 1$. Furthermore for vectors $x,y \in \F_q^d$ let $x\cdot y$ denote $x_1y_1 + x_2y_2 ... +x_dy_d$, the bilinear form inducing our ``norm'' $||\cdot||$. Then we have
\begin{align*}
    ||a|| & = ||\sum_{i=0}^{n} c_iv_i|| = \sum_{j=1}^d \left(\sum_{i=0}^n c_iv_{ji}\right)^2 = \sum_{j=1}^d \ \sum_{0\leq i, k \leq n} c_i c_k v_{ji} v_{jk} = \sum_{0\leq k, l \leq n} \ \sum_{j=1}^d c_i c_k v_{ji} v_{jk}  \\
    & = \sum_{0\leq i,k \leq n} c_ic_k(v_i \cdot v_k) = c_n^2||v_n|| + c_n \sum_{i=0}^{n-1} c_i (v_n \cdot v_i) + \sum_{0\leq i,k \leq n-1} c_i c_k (v_i \cdot v_k).
\end{align*}
We want $||a|| = t$, so once $c_1,...,c_{n-1}$ have been fixed this is an equation of the form 
$$c_n^2 \alpha + c_n \beta + \gamma = t$$ 
for constants $\alpha, \beta, \gamma$. This is quadratic in $c_n$ and has at most two solutions. 
\end{proof}

We now show that, under stronger assumptions on the size of $E$, a positive proportion of nondegenerate prisms are affinely nondegenerate. Here and onward $C_d$ denotes a value that is constant with respect to $q$ but not $d$. Since we are assuming $q\gg d$ such values are essentially constant.  

\begin{lem}\label{almost all prisms are AN}
    Let $N'_d(E)$ be the number of affinely nondegenerate prisms in $E\subseteq \mathbb{F}_q^d$, $d\geq 3$, and assume that $d=3$ or $|E|  \geq C_d q^{d - \frac{1}{d-1}}$. Then 
    \begin{equation}
        \frac{N_d(E) - N'_d(E)}{N_d(E)} \leq C'_d, \quad with \quad C'_d < 1 .
    \end{equation}
    That is, an asymptotically positive proportion of nondegenerate prisms are affinely nondegenerate. In particular, this means by Theorem \ref{prism count},
     \begin{equation}
        N'_d \geq C_d \frac{|E|^{d+2}}{q^{2d}}.
    \end{equation}

\end{lem}
\begin{proof}
    First note that if $d=3$ all nondegenerate prisms are affinely nondegenerate. For a prism to be affinely degenerate in $d=3$ its center would have to lie on a 1-dimensional affine subspace $A$. But then by Lemma \ref{Brian Lemma} if $y$ is a pole of this prism then we have $|A\cap (S_t + y)| \leq 2$. But a nondegenerate prism in $d=3$ must have 3 distinct center points. 

    Now consider the $d>3$ case. Define $k_{(x,y)}$ as in Theorem \ref{prism count}.  In counting affinely nondegenerate prisms, it suffices to find an upper bound for the count of nondegenerate prisms with affinely dependent center.  In such a prism $P=(y,z,x^1,...,x^d)$, fixing the pair $(y,z)$, the number of choices for $(x^1,...,x^{d-1})$ is at most $k_{(y,z)}^{d-1}$, since each $x^i$ must be chosen to be distance $t$ from both $y$ and $z$.  Having chosen $(y,z,x^1,...,x^{d-1})$, if the center $\{x^1,...,x^{d}\}$ is affinely dependent, then $x^d$ must be on the affine subspace 
    $A_0$ of $\mathbb{F}_q^d$ generated by $\{x^1,...,x^{d-1}\}$, which has dimension
    $$
    \text{rank}\left(\begin{array}{cc}
    x^2-x^1 \\
    x^3 - x^1 \\
    \vdots \\
    x^{d-1}-x^1
    \end{array}\right)\leq d-2.  
    $$
    Therefore, by Lemma \ref{Brian Lemma}, $|A_0\cap (y+S_t)|\leq 2q^{d-3}$.  Now we can bound our total count of nondegenerate prisms with affinely dependent centers, using the above calculations and the fact that $k_{(y,z)}\leq 2q^{d-2}$ for any pair $(y,z)$ (which is again by Lemma \ref{Brian Lemma}):
    $$
    N_d(E)-N'_d(E)\leq 2q^{d-3}\sum_{y,z\in E}{\left(k_{(y,z)}\right)^{d-1}}
    \leq 2q^{d-3}\left(2q^{d-2}\right)^{d-2}\sum_{y,z\in E}{k_{(y,z)}}
    $$
    $$
    \lesssim_d q^{d^2-3d+1} \frac{|E|^3}{q^2} =q^{d^2-3d-1}|E|^3.
    $$
    By Theorem \ref{prism count}, $N_d(E)\gtrsim \frac{|E|^{d+2}}{q^{2d}}$.  Therefore, if $|E|\geq C_d q^{d-\frac{1}{d-1}}\log{q}$, then
    $$
    \frac{N_d(E) - N'_d(E)}{N_d(E)}
    \leq C''_d \frac{q^{d^2-3d-1}|E|^3}{\frac{|E|^{d+2}}{q^{2d}}}
    =C_d'' \frac{q^{d^2-d-1}}{|E|^{d-1}}
    \leq C_d' < 1.$$

\end{proof}
We now turn our attention to showing one of these affinely nondegenerate prisms does not admit a bad set. We proceed as follows.

\begin{lem}\label{intersect spheres}
    Suppose that the set of distinct points $\{a_i\}_{i=1}^k$ are affinely independent. That is, the set $\{a_1 - a_j: 2\leq j\leq k\}$ is a linearly independent set of vectors. Then
    \begin{equation}
        \left| \bigcap_{i=1}^k (S_t + a_i) \right| \leq 2q^{d-k}.
    \end{equation}
\end{lem}
\begin{proof}
Note that the set $\bigcap_{i=1}^k (S_t + a_i)$ corresponds to the set of vectors $x$ with $||x||=t$ such that $||a_1 + x - a_j|| = t$ for $2\leq j \leq k$. Let $a'_j = a_j - a_1$ and write $a_j' =  (\alpha_{j,1},..,\alpha_{j,d})$. Then any such $x=(x_1,...,x_d)$ satisfies the system of equations
\begin{align*}
    (x_1 - \alpha_{2,1})^2 + (x_2 - \alpha_{2,2})^2 & + ... (x_d - \alpha_{2,d})^2 = t \\
    (x_1 - \alpha_{3,1})^2 + (x_2 - \alpha_{3,2})^2 & + ... (x_d - \alpha_{3,d})^2 = t \\
    & \vdots \\
    (x_1 - \alpha_{k,1})^2 + (x_2 - \alpha_{k,2})^2 & + ... (x_d - \alpha_{k,d})^2 = t.
\end{align*}
Expanding and noting that $||x||=t$ we obtain the linear system of equations
\begin{align*}
    2x_1\alpha_{2,1} + 2x_2\alpha_{2,2} & + ... + 2x_d\alpha_{2,d} = ||a'_2|| \\
    2x_1\alpha_{3,1} + 2x_2\alpha_{3,2} & + ... + 2x_d\alpha_{3,d} = ||a'_3|| \\
    & \vdots \\
    2x_1\alpha_{k,1} + 2x_2\alpha_{k,2} & + ... + 2x_d\alpha_{k,d} = ||a'_3||.
\end{align*}
Since we assumed the $a_j$s were affinely independent, this system's corresponding matrix has full rank. Thus its solution space $A$ is an affine subspace of dimension $d-k+1$. However we are only interested in those $x \in A$ with $||x|| = t$. This corresponds to the intersection $A\cap S_t$ which by Lemma \ref{Brian Lemma} has cardinality $< 2q^{d-k}$, completing the proof. 
\end{proof}

To complete the proof we will bound the number of prisms that admit a bad set by counting the number of prisms a given set of size $k$ can be bad in. 

\begin{lem}\label{poles size q^a} Suppose that $B$ is a bad set, with $|\Pole(B)| > 2 q^{a-1}$. For every $y, z \in \Pole(B),$ there exists a subset $J \subset \Pole(B)$ such that $J \cup \{y,z\}$ are affinely independent and $|J| = a$. 
\end{lem}
\begin{proof}
Fix $b\in B$ and note that all points in $\Pole(B)$ lie on the sphere $S_t + b$. We build a sequence of sets $J_1 \subset J_2 \subset .... \subset J_a = J$ such that $|J_i| = i$ and each $J_i \cup \{y,z\}$ is affinely independent. Suppose we have chosen $J_i$. Then then we can choose any point for $J_{i+1}\sm J_i$ that does not lie in the $(i+1)$-dimensional affine subspace $A$ spanned by $\{y,z\} \cup J_i$. Since the points we have to choose from lie on $S_t + b$ this rules out the points in $(S_t + b) \cap A = b+(S_t \cap (A-b))$. By Lemma \ref{Brian Lemma} this set has size $\leq 2q^i$. So by assumption there is a point in $\Pole(B)$ we can choose. 
\end{proof}

With all the pieces in place, we can now complete our proof.
\begin{lem}\label{counting bad sets}
Fix some set $B$ with $|B|=k$. Then $B$ is bad in at most $C_dq^{d^2 - kd - d +k - 1}$ affinely nondegenerate prisms. 
\end{lem}

Note that this would suffice to prove our main result in the case $d\geq 3$. To see why this is true, let $M_k(E)$ be the number of affine nondegenerate prisms with affinely independent centers in $E$ that admit a bad set of size $k$, and let $M(E)=\sum_{k=1}^{d-1}{M_k(E)}$. Then we have that for $|E|>q^{d-1}$,
\begin{align}
M(E) &\leq C_d \sum_{k=1}^{d-1} |E|^k q^{d^2 - kd - d + k - 1} \leq C_d (d-1)E^{d-1}q^{d^2 - (d-1)d - d + (d-1)} = C_d (d-1)E^{d-1}q^{d-2} \\&< C_d dE^{d-1}q^{d-2}.
\end{align}
We want to show that $N'_d(E) > M(E)$. Assuming $d=3$ or $|E|>C_d q^{d-\frac{1}{d-1}}$ we have by Lemma \ref{almost all prisms are AN} that $N'_d > C\frac{|E|^{d+2}}{q^{2d}}$. So it suffices to show that 
\begin{equation}
\frac{|E|^{d+2}}{q^{2d}} > C_d d|E|^{d-1}q^{d-2},
\end{equation}
which is true whenever 
\begin{equation}
|E| \geq C_d q^{d-\frac{2}{3}}.
\end{equation}
In the case of $d=3$ this is the strongest bound on $|E|$. Otherwise it is subsumed under the stronger constraint of $|E| > C_d q^{d - \frac{1}{d-1}}$ required for Lemma \ref{almost all prisms are AN}. This completes the proof of Theorem \ref{main thm} when $d\geq 3$, and the $d=2$ case follows immediately from techniques in \cite{FIMW}:  First prune the set $E$, obtaining $E'\subseteq E$ with a positive proportion of the points in $E$, such that every point in $E'$ has large vertex degree in $\mathcal{G}_t(E)$.  In \cite{FIMW} for example, they obtain $|E'|\geq \frac{1}{32}|E|$ where every point in $E'$ is adjacent to at least 100 points in $E$, which is more than sufficient here.  Then apply Lemma 4.1 from that paper, and we have constructed the desired configuration in $\mathbb{F}_q^2$ as long as $|E|\geq Cq^{\frac{7}{4}}$.  To finish proving Theorem \ref{main thm}, it only remains to prove Lemma \ref{counting bad sets}.
\begin{proof}
Consider an affinely nondegenerate prism $P$ with center $\mathcal{C} (P) = \{ x^1, \ldots, x^d \}$ and tail $\mathcal{T} (P) = \{ y,z \}$. Suppose that $B = \{x^1, \ldots, x^k \}$ is $P$-bad. We will count $M_B(E)$, the number of choices for other nondegenerate prisms $Q$ for which $B$ is $Q$-bad. The key observation of the proof is that the tails of $Q$ must be chosen from among the poles of $B$. However, the more poles $B$ has, the more constrained the choices for center points of $Q$ are since the condition of badness requires each pole to be distance $t$ away from at least one center point. 
    
First we bound the size of $\Pole(B)$. Each pole of $B$ must also be distance $t$ away from some other point in $\C(P)$. Since $\C(P)$ is affinely independent we apply Lemma \ref{intersect spheres} and obtain:
\begin{equation}\label{max poles}
\Pole(B) \subset \bigcup_{a \in \C(P) \sm B} \left((S_t + a) \cap \bigcap_{b \in B} S_t + b\right) < 2(d-k)q^{d-k-1}.
\end{equation}
Let $\ell$ be minimal such that $\Pole(B) \leq 2q^{\ell}$. Then we have $\leq 4q^{2\ell}$ choices of tail for $Q$. Fix a choice of tail $\{y,z\}$, and we will count the number of ways to choose the center. By assumption $\Pole(B) > 2q^{\ell-1}$ so by Lemma \ref{poles size q^a} there exists a subset $J \subset \Pole(B)$ with $|J| =\ell$ and $J \cup \{y,z\}$ affinely independent. Choose any such $J$. Let $\phi: E \setminus B \rightarrow \mathcal{P}(J)$ be defined by $\phi(x) = J \cap \Pole(x)$. Consider $A = (a_1,a_2,\dots,a_{d-k}) \in (E \setminus B)^{d-k}$, a tuple with distinct elements. Let $T_A = (\phi(a_1),\phi(a_2),\dots,\phi(a_{d-k})) \in (\mathcal{P}(J))^{d-k}.$ 
    
Suppose that $\mathcal{C}(Q) = B \cup A$. If $B$ is $Q$-bad then 
\begin{equation}\label{union of phi ai}
    \bigcup_{i = 1}^{d-k} \phi(a_i) = J.  
\end{equation}
So we can limit the choices of other center points to only those tuples which fulfill the above condition. That is, we fix ahead of time the values $Y_i = \phi(a_i)$ and count the number of choices of center points that realize those values. Noting that $J\cup \{y,z\}$, we have by Lemma \ref{intersect spheres} that there are $\leq 2q^{d-2-|Y_i|}$ choices for $a_i$. Further note that by Equation \ref{union of phi ai} we have $\sum |Y_i| \geq \ell$. We compute the following:
    \begin{align}
        M_B(E) \leq  4 q^{2\ell} \sum_{\substack{(Y_1, \ldots,Y_{d-k} ) \\ \cup Y_i = J}} \prod_{i=1}^{d-k} 2q^{d - 2 - |Y_i|} &= 4 q^{2\ell} \sum_{\substack{(Y_1, \ldots,Y_{d-k} ) \\ \cup Y_i = J}} (2q^{d-2})^{d-k} \prod_{i=1}^{d-k} q^{- |Y_i|} \\
        &= C_d q^{2\ell} \sum_{\substack{(Y_1, \ldots,Y_{d-k} ) \\ \cup Y_i = J}} (2q^{d-2})^{d-k} q^{\left(- \sum_{i = 1}^{d-k} |Y_i| \right)} \\
        &\leq C_d q^{2\ell} \sum_{\substack{(  Y_1, \ldots,Y_{d-k} ) \\ \cup Y_i = J}} (2q^{d-2})^{d-k} q^{-\ell} \\
        &= C_d \ q^{d^2 - kd - 2d + 2k + \ell},
    \end{align}
    where $C_d$ is the number of $(Y_1, \ldots, Y_{d-k})$ such that $\cup Y_i = J$, which is a constant dependent on $d$. Notice $q^{d^2 - kd - 2d + 2k + \ell}$ is maximized when $\ell$ attains its maximum value. By Equation \ref{max poles}, $\ell = d - k - 1$. Therefore, 
    \begin{align}
        M_B(E) \leq 4q^{2a} \sum_{\substack{(Y_1, \ldots,Y_{d-k} ) \\ \cup Y_i = J}} \prod_{i=1}^{d-k} q^{d - 2 - \alpha_i} &\lesssim q^{d^2 - kd - d + k - 1 }.
    \end{align}
    As the number of bad sets of size $k$ is $\leq |E|^k$, we have that 
    \begin{equation}
       M_k(E) \ \lesssim \ |E|^k q^{d^2 - kd - d + k - 1 }.
    \end{equation}
\end{proof}  

\section{Conclusion}
\subsection{Connection to the Single-Parameter Case}
Note that the construction we use to shatter $d$ points also suffices to shatter $d$ points using $\Hh_t^{'d}$, the single-parameter classifiers studied in \cite{FIMW}. Indeed if $A\subset \C(P)$ is not $P$-bad then by definition we can find a point $y$ that is a pole of $A$ but is not distance $t$ away from any other point of $\C(P)$, and so the classifier $h_y \in \Hh_t^{'d}$ restricted to $\C(P)$ is the indicator function on $A$. And so we have that the VC-dimension of $\Hh_d^{'t}$ is at least $d$ provided $E$ is large enough. 

However, this exact construction cannot work to shatter $d+1$ points since doing so would also involve shattering $d+1$ with our two-parameter classifiers, which is impossible. One might wonder whether a slightly different construction might work, where instead of looking for prisms we look for sets $(z,x^1,...,x^{d+1})$ where $z$ is the only common pole of the $x^i$, a sort of ``star.'' The number of these stars could be counted using Theorem \ref{chain count} and the same technique as Theorem \ref{prism count}, just replacing 2-paths with 1-paths. However an issue arises comes with counting the number of stars a $d$-set can be bad in. Our pigeonholing argument works because the condition of badness reduces the number of poles a bad $k$-set can have by a power of $q$, while also restricting the number of prisms such a set can be bad in if it has many poles. But by Lemma \ref{intersect spheres}, an affinely independent set of size $d$ has at most $2$ poles, and thus no such restriction could exist. 

\subsection{Future Work}
There are a number of possible directions for future work. One would be attempting to take our results further, in the sense of improving the exponent constraining the size of $E$. Our proof required showing that a positive proportion of nondegenerate prisms are affinely nondegenerate, which placed a very strong constraint on $|E|$ in the $d>3$ case. Were an approach to be found that did away with this requirement or weakened this constraint, our bound could likely be improved. 

Another direction would be trying to obtain similar results for other sets of classifiers on subsets of $\F_q^d$. We obtained our classifiers from those in \cite{FIMW} by adding an additional parameter; one could consider adding even more parameters. We suspect this case could be fairly easily resolved by similar techniques to those used here, but the problem could be changed further. Finite field VC-dimension problems such as this are relatively unexplored, so there are many different avenues to pursue.


\end{document}